\documentclass[10pt]{amsart}

\usepackage{scalerel}
\usepackage{mathrsfs}
\usepackage{bigints}
\usepackage{amsmath}
\usepackage{hyperref} 
\usepackage{amsfonts}
\usepackage{amsthm}
\usepackage{amssymb}
\usepackage{amsmath}
\usepackage{graphicx}
\usepackage[utf8]{inputenc}
\usepackage[english]{babel}
\usepackage{placeins}
\usepackage[mathscr]{eucal}
\usepackage{mathptmx}

\DeclareMathOperator{\ad}{ad}
\DeclareMathOperator{\Tr}{Tr}
\DeclareMathOperator{\e}{e}
\DeclareMathOperator{\cc}{c}
\DeclareMathOperator{\Rank}{rank (G\,/\,K)}
\DeclareMathOperator{\supp}{supp}

\newtheorem{theorem}{Theorem}[section]
\newtheorem*{acknowledgment*}{Acknowledgment}
\newtheorem{lemma}{Lemma}[section]

\newtheorem*{theorem*}{Theorem}
\newtheorem{corollary}{Corollary}[section]
\newtheorem{proposition}{Proposition}[section]
 
\begin{document}
\thispagestyle{empty}

\title[Radon-Nikodym Derivative of a Convolution of Orbital Measures]{Regularity of the Radon-Nikodym Derivative of a Convolution of Orbital Measures on Noncompact Symmetric Spaces}

\author{Boudjem\^aa Anchouche}
\address{Department of Mathematics, College of Science, Kuwait University, P. O. Box 5969, 13060 Safat, Kuwait}
 \email{anchouch@sci.kuniv.edu.kw}

\begin{abstract}
	Let $G/K$ be a Riemannian symmetric space of noncompact type, and let $\nu_{a_j}$, $j=1,...,r$ be some orbital measures on $G$ (see the definition below). The aim of this paper is to study the $L^{2}$-regularity (resp. $C^k$-smoothness) of the Radon-Nikodym derivative of the convolution $\nu_{a_{1}}\ast...\ast\nu_{a_{r}}$ with respect to a fixed left Haar measure $\mu_G$ on $G$. As a consequence of a result of Ragozin, \cite{ragozin}, we prove that if $r \geq \, \max_{1\leq i \leq s}\dim {G_i}/K_i$, then $\nu_{a_{1}}\ast...\ast\nu_{a_{r}}$ is absolutely continuous with respect to $\mu_G$, i.e., $d\big(\nu_{a_{1}}\ast...\ast\nu_{a_{r}}\big)/d\mu_G$ is in $L^1(G)$,  where $G_i/K_i$, $i=1,...,s$, are the irreducible components in the de Rham decomposition of $G/K$. The aim of this paper is to prove that $d\big(\nu_{a_{1}}\ast...\ast\nu_{a_{r}}\big)/d\mu_G$ is in $L^2(G)$ (resp. $C^k\left(G \right) $) for $r \geq \max_{1\leq i \leq s}\dim \left( {G_i}/{K_i}\right)  + 1$\, (resp. $r \geq \max_{1\leq i \leq s} \dim\left( {G_i}/{K_i}\right)  +k+1$). The case of a compact symmetric space of rank one was considered in \cite{AGP} and \cite{AG}, and the case of a complex Grassmannian was considered in \cite{AA}.
\end{abstract}


\maketitle

\tableofcontents

\footnote{This paper was uploaded on Researchagate on August 2018, DOI: 10.13140/RG.2.2.34657.97122}
\section{Introduction}\label{introduction}
Let $G$ be a real, connected, noncompact semisimple Lie group with finite center, and $K$ a maximal compact subgroup of $G$, hence $G/K$ is a symmetric space of noncompact type. Unless otherwise stated, we assume in all what follows that $G/K$ is irreducible.
Let $a_1$, $\cdots$, $a_r$ be points in $G-N_G\left(K \right) $, where $N_G\left(K \right)$ is the normalizer of $K$ in $G$. For each integer $j$, $1 \leq j \leq r$, let
\[
{\mathscr{I}}_{a_j}(f)=\int_K \int_K f(k_1a_jk_2)  d\mu_K(k_1)d\mu_K(k_2)
\]
where $f$ is a continuous function with compact support in $G$ and $\mu_K$ a normalized Haar measure on $K$.

To the linear functional ${\mathscr{I}}_{a_j}$ corresponds, by Riesz representation theorem, see \cite{Conway}, a measure, which will be denoted by $\nu_{a_j}$, i.e., there exists a Borel measure on $G$ such that
\[
{\mathscr{I}}_{a_j}(f)=\int_Gf(g)d\nu_{a_j}(g).
\]
Since
\[{\mathscr{I}}_{a_j}={\mathscr{I}}_{k_1a_jk_2},
\]
for any $k_1, k_2$ in $K$, the measure $\nu_{a_j}$ is $K$-bi-invariant.

The aim of this paper is to study the regularity of the Radon-Nikodym derivative of the convolution $\nu_{a_{1}}\ast...\ast\nu_{a_{r}}$  with respect to a fixed Haar measure $\mu_{\mathsf{G}}$ of $G$. More precisely, the aim is to prove the following

\begin{theorem*}[Main Theorem]\label{Main Theorem} Let $G/K$ be an irreducible symmetric space of noncompact type, $a_1$, ..., $a_r$ points in $G-N_G(K)$, $\nu_{a_1}$,...,$\nu_{a_r}$ be the associated orbital measures. 
\begin{enumerate}
\item
If 
	\[
	r \geq \dim\, G/K +1,
	\]
	 then
	\[
	\frac{d\left( \nu_{a_{1}}\ast...\ast\nu_{a_{r}}\right) }{d\mu_G} \in L^2\left( G\right).
	\]		
	\item
	If 
	\[
	r \geq \dim\, G/K + k+1,
	\]
	 then
	\[
	\frac{d\left( \nu_{a_{1}}\ast...\ast\nu_{a_{r}}\right) }{d\mu_G} \in C^k \left( G\right).
	\]		
	\end{enumerate}
\end{theorem*}	

The paper is organized as follows:  Section \ref{prelim} consists of some preliminary results. In section \ref{spher-trans}, we state some results on spherical Transform of the density function. In section \ref{L2-regularity}, we investigate the $L^2$-regularity of the Radon-Nikodym derivative. In section \ref{$C^k$-regularity} we study the smoothness of the Radon-Nikodym derivative. In section \ref{general-case} we consider the case of a reducible symmetric space of noncompact type.

\section{Some Preliminary Results} \label{prelim}

Let $G$ be a real, connected, noncompact semisimple Lie group with finite center, and $K$ a maximal compact subgroup of $G$.  Fix a Cartan involution $\theta$ and let ${\mathfrak{g}}={\mathfrak{k}} \oplus {\mathfrak{p}} $ be the corresponding Cartan decomposition of the Lie algebra ${\mathfrak{g}}$ of $G$, where ${\mathfrak{k}}$ is the Lie algebra of $K$ and $\mathfrak{p}$ is the orthogonal complement of $\mathfrak{k}$ with respect to the Killing form of $\mathfrak{g}$. It is well known that $G/K$ has a structure of a symmetric space of noncompact type, where the metric is induced from the restriction of the Killing form of ${\mathfrak{g}}$ to ${\mathfrak{p}}$.
Let $\mathfrak{a}$ be a maximal abelian subspace of $\mathfrak{p}$, ${\mathfrak{a}}^*$ its dual, and ${\mathfrak{a}}_{\mathbb{C}}^*$ the complexification of ${\mathfrak{a}}^*$, i.e., ${\mathfrak{a}}_{\mathbb{C}}^*$ is the set of linear $\mathbb{R}$ forms on ${\mathfrak{a}}$ with values in $\mathbb{C}$.
The dimension of $\mathfrak{a}$ is independent of the choice of a Cartan decomposition and $\dim \mathfrak{a}$ is called the rank of the symmetric space $G/K$ and denoted by $l=\Rank$.  

\medskip

The Killing form $B$ of $\mathfrak{g}$ is non-degenerate on $\mathfrak{a}$, so it induces an isomorphism between $\mathfrak{a}$ and ${\mathfrak{a}}^*$. The extension of the inner product on $\mathfrak{a}$ induced by the Killing form to ${\mathfrak{a}}_{\mathbb{C}}^*$ will be denoted also by $\left\langle .,.\right\rangle $.
For an element $\lambda \in {\mathfrak{a}}_{\mathbb{C}}^*$, we denote by $H_{\lambda}$ the corresponding element in ${\mathfrak{a}}_{\mathbb{C}}$., i.e., by Riesz Theorem, there exist $H_{\lambda}$ in $\mathfrak{a}$ such that $\lambda(H)=\left\langle H, H_{\lambda}\right\rangle$, for all $H \in {\mathfrak{a}}_{\mathbb{C}}$.
 Under this correspondence, ${\mathfrak{a}}^*$ corresponds to $\mathfrak{a}$. We transfer the inner product defined on ${\mathfrak{a}}_{\mathbb{C}}$ to an inner product on ${\mathfrak{a}}_{\mathbb{C}}^*$, via $\left\langle \lambda, \mu\right\rangle :=\left\langle H_\lambda, H_\mu\right\rangle$. 

For $\alpha$ in ${\mathfrak{a}}^*$, we put
\[
{\mathfrak{g}}_{\alpha}=\Big\{X \in {\mathfrak{g}} \mid \ad(H)X=\alpha(H)X, \text{ for all } H  \text{ in } {\mathfrak{a}}\Big\}.
\]
A nonzero element $\alpha$ in ${\mathfrak{a}}^*$ is called a restricted root if ${\mathfrak{g}}_{\alpha} \neq 0$. So we have
\[
{\mathfrak{g}}={\mathfrak{g}}_0\oplus \sum_{\alpha \in \Sigma}{\mathfrak{g}}_{\alpha},
\]
where 
\[
{\mathfrak{g}}_0=\Big\{X \in {\mathfrak{g}} \mid \ad(H)X=0, \text{ for all } H  \text{ in } {\mathfrak{a}}\Big\}.
\]

Denote by $\Sigma$ the set of restricted roots on $\mathfrak{a}$, and let
\[
{\mathfrak{a}}^{'}=\Big\{X \in {\mathfrak{a}} \mid \alpha\left( X\right) \neq 0, \text{ for all } \alpha  \text{ in } \Sigma\Big\}.
\]
The connected components of ${\mathfrak{a}}^{'}$, which are open convex sets, are called Weyl chambers. Let $M'$ (resp. M) be the normalizer (resp. centralizer) of $A$ in $K$. Then the group $\mathcal{W}=M'/M$, called the Weyl group, is a finite group acting transitively on the set of Weyl chambers. The induced action of $\mathcal{W}$ on $\mathfrak{a}^*$ is given by $w\lambda (H)=\lambda(w^{-1}H)$. 

\medskip

Fix a connected component $\mathfrak{a}^+$ of $\mathfrak{a}'$, and call it a positive Weyl chamber, and let
\[
\Sigma^+=\Big\{\alpha \in \Sigma \mid \alpha\left( X\right)>0, \text{ for all } X  \text{ in } {\mathfrak{a}^{+}}\Big\}, \text{ and   }\, \,  \Sigma^-=\Big\{-\alpha \mid \alpha \in \Sigma^+ \Big\}.
\]

The set $\Sigma^+$ (resp. $\Sigma^-$) is called the set of positive (resp. negative) restricted roots with respect to the Weyl chamber ${\mathfrak{a}}^+$. For $\alpha$ in $\Sigma$, we put  $m_{\alpha}=\dim {\mathfrak{g}}_{\alpha}$, and let
\[
\varrho = \frac{1}{2}\Tr{\ad_{\mathfrak{n}}}_{\mid \mathfrak{a}}=\frac{1}{2}\sum_{\alpha \in \Sigma^+}m_{\alpha}\alpha, \hspace{0.5cm}{  } \mathfrak{n} =\sum_{\alpha \in \Sigma^+}{\mathfrak{g}}_{\alpha}.
\]
Let $A, K$, and $N$ be Lie subgroups of $G$ with Lie algebras $\mathfrak{a}$, $\mathfrak{k}$, and $ \mathfrak{n}$. Then we have the Iwasawa decomposition 
\[
G=KAN=K\exp\left( \mathfrak{a}\right)N.
\]
The Iwasawa projection 
\[
H:G \longrightarrow \mathfrak{a}
\]
is the map which to each $g$ in $G$ associates the unique element $H(g)$ in $\mathfrak{a}$ such that $g \in K\exp(H(g))N$.
Each element $g$ in $G$ can be written, 
\[
g=k\left( g\right)\exp \big(H\left( g\right)\big)n\left( g\right),
\]
where $k\left( g\right) \in K$, and $n\left( g\right) \in N$.
We have also the Cartan decompostion
\[
G=KAK,
\]
or more precisely, the decomposition
\[
G=K\overline{\exp(\mathfrak{a}^+)}K,
\]
where $\overline{\exp(\mathfrak{a}^+)}$ is th closure of $\exp(\mathfrak{a}^+)$. So every $K$-bi-invariant function can be considered as a function on $A$ or a function on $\overline{\exp(\mathfrak{a}^+)}$.
The following result of Harish-Chandra gives a characterization of spherical functions on $G$. 

\begin{theorem} \cite[Harish-Chandra]{GV} \label{harish-chandra-rep} The spherical functions on $G$ are parametrized by $\mathfrak{a}^*_{\mathbb{C}}$. More precisely, for each $\lambda$ in ${\mathfrak{a}}_{\mathbb{C}}^{*}$ corresponds a spherical function $\varphi_{\lambda}$ on $G$ given by
\[
\varphi_{\lambda}\left(g \right) =\int_K\e^{(\sqrt{-1}\lambda - \varrho)\left( H(gk)\right)} d\mu_K\left( k\right).  
\]
Moreover, $\varphi_{\lambda}=\varphi_{\nu}$ if and only if $\lambda$ and $\nu$ are in the same orbit under the action of the Weyl group, i.e., there exists $s$ in the Weyl group $\mathcal{W}$ such that $\lambda =s\nu$.
\end{theorem}

Let $L^1(K\diagdown G \diagup K)$ be the space of $K$-bi-invariant $L^{1}$ functions on $G$. 
The Harish-Chandra transform  of a function $f$ in $L^1(K\diagdown G \diagup K)$, denoted by  ${\mathcal{H}}(f)$, is given by
\[
{\mathcal{H}}(f)\left( \lambda\right) =\int_G f\left( g\right) \varphi_{\lambda}\left( g^{-1}\right)d\mu_G(g). 
\]
It is known that if $\lambda$ is in $\mathfrak{a}^*$, then the $(G,K)$ spherical function $\varphi_{\lambda}$ is positive definite, hence bounded, \cite{Wolf}, Corollary 11.5.11 \& Proposition 8.4.2 (i). Therefore ${\mathcal{H}}(f)(\lambda)$ is well defined  for $f$ in $L^1(G)$ and $\lambda$ in $\mathfrak{a}^*$. The Harish-Chandra transform gives a map from $L^1\left( K\diagdown G\diagup K\right) $ to the $\mathcal{W}$-invariant function on ${\mathfrak{a}}^*$. \\

Let $B(K\diagdown G \diagup K)$ be the space of all linear combinations of continuous positive definite $K$-bi-invariant functions $f:G \longrightarrow \mathbb{C}$. Then we have the following
\begin{theorem}\cite[Theorem 11.5.26, (Plancherel Theorem)]{Wolf} \label{planch} For $f \in B(K\diagdown G \diagup K)\, \cap \,L^1(K\diagdown G \diagup K)$,
	\[
	\int_G\left|f\left(g \right)  \right|^2d\mu_G \left(g \right) ={\frac{1}{\left| \mathcal{W}\right| }}\int_{\mathfrak{a}^*}\left| {\mathcal{H}}(f)\left( \lambda\right)\right|^2
	\left|\cc\left(\lambda \right)  \right|^{-2} d\,\lambda
	\]
	where $\left| \mathcal{W}\right|$ is the number of elements of the Weyl group $\mathcal{W}$, $\cc$ is the Harish-Chandra function.
\end{theorem}

The following inversion for the spherical transform is also needed in this paper

\begin{theorem}\cite[Theorem $11.5.26$, (Inversion Formula)]{Wolf} \label{inversion} For $f \in B(K\diagdown G \diagup K)\, \cap \, L^1(K\diagdown G \diagup K)$,
	\[
	f\left(g \right) ={\frac{1}{\left| \mathcal{W}\right| }}\int_{\mathfrak{a}^*} {\mathcal{H}}(f)\left( \lambda\right) \varphi_{\lambda}\left(g\right)
	\left|\cc\left(\lambda \right)  \right|^{-2} d\,\lambda
	\]
	where $\left| \mathcal{W}\right|$ is the number of elements of the Weyl group $\mathcal{W}$, $\cc$ is the Harish-Chandra function.
\end{theorem}

\section{Spherical Transform of the Density Function} \label{spher-trans}

The notations are as in section \ref{prelim}. Let $\varphi_{\lambda}$ be the spherical function for the Gelfand pair $\left(G, K \right) $ corresponding to $\lambda$ in ${\mathfrak{a}}^*$. As was mentioned above, the spherical, or Harish-Chandra, transform of a function $f$ in $L^1\left( G\right)$ is defined by
\[
{\mathcal{H}}(f)\left( \lambda\right) =\int_G f\left( g\right) \varphi_{\lambda}\left( g^{-1}\right)d\mu_G(g). 
\]
We define the spherical transform of a compactly supported measure $\mu$ by
\[
{\mathcal{H}}(\mu)\left( \lambda\right) =\int_G  \varphi_{\lambda}\left( g^{-1}\right)d\mu(g).
\]
It is clear that if $\mu$ is absolutely continuous with respect to a fixed left Haar measure $\mu_G$ of $G$, i.e., $d\,\mu = f d\, \mu_G$, then
\begin{equation}
{\mathcal{H}}(\mu)={\mathcal{H}}(f).\label{equal-fonct-measure}
\end{equation}
To simplify the notation, we denote $\nu_{a_j}$ simply by $\nu_j $ and hence, denote the convolution $\nu_{a_{1}}\ast...\ast\nu_{a_{r}}$ by $\nu_{1}\ast...\ast\nu_{r}$.
\begin{proposition}\label{widehatofprod}
	\[
	{\mathcal{H}\big(\nu_{1}\ast...\ast\nu_{r}\big)(\lambda)}=\prod_{i=1}^r \varphi_{\lambda}(a_i^{-1}).
	\]		
\end{proposition}	

\begin{proof}
	Let $r=1$. Then
	\begin{align*}
	{\mathcal{H}}\big(\nu_{1}\big)(\lambda) & =\int_{G}\varphi_{\lambda}(g^{-1}) d\nu_{1}(g)  \\
	&=\int_K \int_K \varphi_{\lambda}\big((k_1a_1k_2)^{-1}\big)d\mu_K(k_1) d\mu_K(k_2).\\
	&=\varphi_{\lambda}(a_1^{-1}) \, \, \, (\text { since $\varphi_{\lambda}$ is $K$-bi-invariant and $\mu_K(K)=1$}).
	\end{align*}
	
	Consider the case $r=2$, i.e, the spherical transform of $\nu_{1}\ast \nu_{2}$. 
	\begin{align}
	{\mathcal{H}}\big(\nu_{1}\ast \nu_{2}\big)(\lambda) & = \int_{G}\varphi_{\lambda}(g^{-1}) d\left( \nu_{1}\ast \nu_{2}\right)(g)   \nonumber\\
		&=  \int_G \int_G \varphi_{\lambda}(g_2^{-1}g_1^{-1})d\nu_{1}(g_1) d\nu_{2}(g_2)  \nonumber\\
	&=\int_G\left(\int_K \int_K  \varphi_{\lambda}(g_2^{-1}k_2^{-1}a_1^{-1}k_1^{-1})d\mu_K(k_1)d\mu_K(k_2)\right) d\nu_{2}(g_2)  \nonumber\\
	&=\int_G\left(\int_K \int_K  \varphi_{\lambda}(g_2^{-1}k_2^{-1}a_1^{-1})d\mu_K(k_1)d\mu_K(k_2)\right) d\nu_{2}(g_2) \label{123a} \\
	&=\varphi_{\lambda}(a_1^{-1})\int_G \varphi_{\lambda}(g_2^{-1}) d\nu_{2}(g_2) \label{123b}\\
	&=\varphi_{\lambda}(a_1^{-1})\int_K \int_K  \varphi_{\lambda}(k_2^{-1}a_2^{-1}k_1^{-1})d\mu_K(k_1)d\mu_K(k_2) \label{123c}\\
	&=\varphi_{\lambda}(a_1^{-1})\varphi_{\lambda}(a_2^{-1}).\label{123d} 
	\end{align}
	
	To get $(\ref{123b})$ from $(\ref{123a})$, we used the fact that $\varphi_{\lambda}$ satisfies
	\[
	\int_K\varphi_{\lambda}\left(xky \right) d\mu_K(k)=\varphi_{\lambda}(x)\varphi_{\lambda}(y),
	\]
	and to get $(\ref{123d})$ from $(\ref{123c})$, we used the fact  $\varphi_{\lambda}$ is $K$-bi-invariant.\\
	The argument goes by induction for arbitrary $r$. 
\end{proof}	

It is easy to see that the measures $\nu_{{1}},...,\nu_{{r}}$ are supported
on $Ka_{1}K,...,$ $Ka_{r}K$ and, from \cite{AA}, we know that the measure $\nu_{{1}}\ast...\ast\nu_{{r}}$	is absolutely continuous with respect to the Haar measure of the group $G$ if and only if the set $Ka_1K\dots Ka_rK$ is of non-empty interior.

\medskip
  Suppose that $G/K$ is an irreducible symmetric space, hence the linear isotropy representation of $K$ on the tangent space $T_{eK}\left( G/K\right) \simeq {\mathfrak{g}}/ {\mathfrak{k}}$ is irreducible and non trivial. Then by Theorem 2.5 in \cite{ragozin}, if $r \geq \dim G/K$, then $ \nu_{1}\ast...\ast\nu_{r}$ is absolutely continuous with respect to $\mu_G$. If we denote by $\varrho_{a_1,\cdots,a_r}$ the Radon-Nikodym derivative of $\nu_{1}\ast...\ast\nu_{r}$ with respect to the Haar measure $\mu_G$ of $G$, then
\[
\varrho_{a_1,\cdots,a_r}= \frac{d\left( \nu_{1}\ast...\ast\nu_{r}\right) }{d\mu_G} \in  L^{1}(G).
\]
From what was said above, the function $\varrho_{a_1,\cdots,a_r}$ is $K$-bi-invariant, i.e.,
\[\varrho_{a_1,\cdots,a_r}\left( k_1gk_2\right)= \varrho_{a_1,\cdots,a_r}\left(g\right), \,\, \text{ for all $k_1$ and $k_2$ in $K$.}
\]
Moreover, from 
\[
\supp \big(\varrho_{a_1,\cdots,a_r}\big) = \supp \bigg(\nu_{1}\ast...\ast\nu_{r} \bigg)=Ka_{1}K...Ka_{r}K,
\]
we see that $\varrho_{a_1,\cdots,a_r}$ is compactly supported. 
In what follows, we will denote by $L^{p}\left( {K\diagdown G\diagup K}\right) $ the space of $K$-bi-invariant functions which are in $L^p(G)$. Hence we have the following
\begin{proposition}  If $r \geq \dim G\diagup K$, then $\nu_{1}\ast...\ast\nu_{r}$ is absolutely continuous with respect to $\mu_G$ and its density function $\varrho_{a_1,\cdots,a_r}$ is in $L^1\left( {K\diagdown G\diagup K}\right)$.
\end{proposition}

\section{$L^2$-regularity of the Radon-Nikodym derivative} \label{L2-regularity}

\begin{theorem}\label{prowwww} Let $a_i=\exp(H_i)$, where $H_i$ is in $\mathfrak{a}^+$ for $i=1,...,r$. If $ r \geq \dim \, G/K + 1$, then $\varrho_{a_1,\cdots,a_r}$ is in $L^2\left(G\right) $.
\end{theorem}

To prove the theorem we need some preparatory results.

\begin{lemma}\label{roots-lambda}
	 There exists a positive constant $c$ such that for all $y$ in $\mathfrak{a}^*$, $y \neq 0$, there exist $i$, $1\leq i\leq r$ such that
	\begin{equation}\label{ineq-import}
	\left| \left<\frac{y}{\left\|y\right\|}, \alpha_i \right> \right| \geq c.
	\end{equation}

\end{lemma}
\begin{proof}
	
	Suppose that the inequality (\ref{ineq-import}) is not true, then there exists a sequence $(y_p)_{p\geq 1}$ in $\mathfrak{a}^*$, $y_p \neq 0, \forall p \geq 1$, such that for all $i$, $1\leq i\leq r$ 
	\[
	\left| \left<\frac{y_p}{\left\|y_p\right\|}, \alpha_i \right> \right| < \frac{1}{p}, \, \forall \, p \geq 1.
	\]
	Since the sequence $\frac{y_p}{\left\|y_p\right\|}$ is in the unit sphere $\mathbb{S}^{l -1}$ in $\mathfrak{a}^*$, after extracting a subsequence, if necessary, we can assume without loss of generality that the sequence $\frac{y_p}{\left\|y_p\right\|}$ converges to $u \in \mathbb{S}^{l-1}$. Hence
	\[
	\left<u, \alpha_i \right>  =0.
	\]
	Since $\{\alpha_i\}_{i=1}^{r}$ is a basis for $\mathfrak{a}^*$, there exist real numbers $c_i$, $i=1,...,r$ such that
	\[
	u=\sum_{i=1}^{r} c_i\alpha_i.
	\]
	Hence
	\[
	\left\|u\right\|^2= \sum_{i=1}^{r} c_i\left<u,\alpha_i\right>=0.
	\] 
	A contradiction, since $u \in \mathbb{S}^{l-1}$.  Therefore, the inequality (\ref{ineq-import}) is established.
\end{proof}

Let $\left( a_j^{-1}\right)_j$ be a sequence of points contained in a compact subset $\mathcal{C}$ of $A$, such that the $a_j \not\in N_G(K)$ for all $j$. For an element $w$ of $\mathcal{W}$, we put
\[
\Sigma^{+}_{w}(\mathcal{C})=\bigg\{\alpha \in\Sigma^{+} \mid w\alpha\big(\log a\big) \neq 0\, \,  \text{ for all } a \in \mathcal{C}\bigg\}.
\]
The  following result is implicit in \cite{DKV}.
\begin{proposition}\label{DKV} For each $j$, there exists a positive constant $C(a_j)$ such that for all $\lambda \in \mathfrak{a}^*$,
	\[
	\left|\varphi_{\lambda}\left(a_j^{-1} \right)\right| \leq  C(a_j)  \sum_{w \in \mathcal{W}} \prod_{\alpha \in \Sigma^{+}_{w}(\mathcal{C})} \bigg(1+\left| \left<\lambda, \alpha\right>\right|\bigg)^{-\frac{1}{2} m_{\alpha}}.
	\]
\end{proposition}
\begin{proof} For $\lambda$ in $\mathfrak{a}^*$, take $F_{a, H_{\lambda}}(k)= \e^{\sqrt{-1}\left<H(ak), H_{\lambda}\right>}$, where $H_{\lambda}$ is defined by  $\lambda(H)=\left<H_{\lambda}, H\right>$ for all $H$ in $\mathfrak{a}$, and let $g(k)= \e^{-\rho\left(H(ak)\right)}$. Then, by Theorem \ref{harish-chandra-rep}, we have
\[
\varphi_{\lambda}(a)= \int_K\e^{\left( \sqrt{-1}\lambda -\rho\right) H(ak)}dk=\int_K\e^{\sqrt{-1}F_{a, H_{\lambda}}(k)}g(k)dk.
\]	
The proposition follows from \cite[Theorem 11.1]{DKV}. 	
\end{proof}

As a consequence of Proposition \ref{DKV} we have the following:

\begin{corollary}\label{cor-est}
Let $a_i$, $i=1,...,r$ be elements of $A$ such that $\Sigma^{+}_{w}(\mathcal{C})=\Sigma^{+}$  for all $w$ in $\mathcal{W}$.	Then there exist  positive constants $\widetilde{C(a_j)}=\left|\Sigma^+ \right| C(a_j)$, with $C(a_j)$ as in Proposition \ref{DKV} such that 
\[
\left|\varphi_{\lambda}\left(a_j^{-1} \right)\right| \leq  \widetilde{C(a_j)} \prod_{\alpha \in \Sigma^{+}} \bigg(1+\left| \left<\lambda, \alpha\right>\right|\bigg)^{-\frac{1}{2} m_{\alpha}}.
\]	
\end{corollary}	
 
\begin{lemma}\cite[Lemma 9.3.3]{Wolf} \label{density}
	$B(K\diagdown G \diagup K) \cap L^1(K\diagdown G \diagup K)$ is dense in $L^1(K\diagdown \, G\, \diagup K)$.
\end{lemma}

Since $\varrho_{a_1,\cdots,a_r}$ is in $L^1(K\diagdown \, G\, \diagup K)$, by Lemma \ref{density}, there exist a sequence $\left( \varrho_{a_1,\cdots,a_r}^{j}\right)_j$ in $B(K\diagdown G \diagup K) \cap L^1(K\diagdown G \diagup K)$ converging in $L^1$ to $\varrho_{a_1,\cdots,a_r}$. Since $\varrho_{a_1,\cdots,a_r}$ is of compact support, we can assume that $\varrho_{a_1,\cdots,a_r}^{j}$ is of compact support. Then, we have the following
\begin{proposition} \label{conv}
	$\lim_{j \rightarrow \infty} {\mathcal{H}\big(\varrho^j_{a_1,\cdots,a_r}\big)(\lambda)}  = {\mathcal{H}\big(\nu_{1}\ast...\ast\nu_{r}\big)(\lambda)}$. 
\end{proposition}
\begin{proof}
	Since
	\[
	{\mathcal{H}\big(\varrho^j_{a_1,\cdots,a_r}\big)(\lambda)} = \int_G \varrho^j_{a_1,\cdots,a_r}\left( g\right) \varphi_{\lambda}\left( g^{-1}\right)d\mu_G(g), 
	\]
	and since $\varphi_{\lambda}$ is bounded for $\lambda$ in $\mathfrak{a}^*$, we get
	\begin{align*}
	\left| {\mathcal{H}\big(\varrho^j_{a_1,\cdots,a_r}\big)(\lambda)}  - {\mathcal{H}\big(\nu_{1}\ast...\ast\nu_{r}\big)(\lambda)}\right| &=\left|\int_G \bigg(\varrho^j_{a_1,\cdots,a_r}\left( g\right) - \varrho_{a_1,\cdots,a_r}\left( g\right)\bigg) \varphi_{\lambda}\left( g^{-1}\right)d\mu_G(g)  \right|\\
	&\leq c  \int_G \left|\bigg(\varrho^j_{a_1,\cdots,a_r}\left( g\right) - \varrho_{a_1,\cdots,a_r}\left( g\right)\bigg)\right| d\mu_G(g).
	\end{align*}
	The Lemma follows from $\bigints_G \left|\bigg(\varrho^j_{a_1,\cdots,a_r}\left( g\right) - \varrho_{a_1,\cdots,a_r}\left( g\right)\bigg)\right| d\mu_G(g) \longrightarrow 0$. 
\end{proof}

\begin{proof} [Proof of Theorem \ref{prowwww}]
Plancherel Theorem applied to $ \varrho_{a_1,\cdots,a_r}^{j}$ says that
\begin{equation}
\int_G\left|\varrho_{a_1,\cdots,a_r}^{j}\left(g \right)  \right|^2d\mu_G \left(g \right) ={\frac{1}{\left| \mathcal{W}\right| }}\int_{\mathfrak{a}^*}\left| {\mathcal{H}}(\varrho_{a_1,\cdots,a_r}^{j})\left( \lambda\right)\right|^2
\left|\cc\left(\lambda \right)  \right|^{-2} d\,\lambda. \label{j-plancherel}
\end{equation}
Combining Proposition \ref{widehatofprod} and Proposition \ref{conv}, we get
\[
\left| {\mathcal{H}}(\varrho_{a_1,\cdots,a_r}^{j})\left( \lambda\right)\right|^2
\left|\cc\left(\lambda \right)  \right|^{-2} \longrightarrow \left| \prod_{i=1}^r \varphi_{\lambda}(a_i^{-1})\right|^2
\left|\cc\left(\lambda \right)  \right|^{-2}.
\]
By Proposition 7.2 in \cite{Helgason-GGA}, there exists a positive constant $C$ such that
\begin{equation}
\left| \cc \left(\lambda \right)\right|^{-2} \leq C \big(1+\left\| \lambda \right\|  \big)^{ n-l},\label{harish-chandra-c-function} 
\end{equation}
for all $\lambda \in \mathfrak{a}^*$,	where $n=\dim \,G/K$ and $l=\Rank$.
Combining (\ref{harish-chandra-c-function}), Lemma \ref{roots-lambda}, and Corollary \ref{cor-est}, we get
\[
\left| \prod_{i=1}^r \varphi_{\lambda}(a_i^{-1})\right|^2
\left|\cc\left(\lambda \right)  \right|^{-2} \leq C \frac{\big(1+\left\| \lambda \right\|  \big)^{ n-l}}{\big(1+c\left\| \lambda \right\|  \big)^{ r}},
\]
where $C$ is a positive constant, and $c$ is the constant which appears in $\left( \ref{ineq-import}\right) $.
Hence 
\[
\left| \prod_{i=1}^r \varphi_{\lambda}(a_i^{-1})\right|^2
\left|\cc\left(\lambda \right)  \right|^{-2} \text{ is in } L^1(\mathfrak{a}^*) \text{ for } r>n.
\]
Moreover, since $\varrho_{a_1,\cdots,a_r}^{j}$ is of compact support, by the Paley-Wiener Theorem for spherical functions on semisimple Lie groups (\cite{Gangolli}, Theorem 3.5), for each positive integer $N$, there exist a constant $C_{N,j}$ such that
\[
\left| {\mathcal{H}\big(\varrho^j_{a_1,\cdots,a_r}\big)(\lambda)}\right|  \leq C_{N,j} \left(1+\left\| \lambda \right\|  \right)^{-N}, \text{ for arbitrary $\lambda$ in $\mathfrak{a}^*$}. 
\]
Then, we can chose $N$ such that 
\[
\left(1+\left\| \lambda \right\|  \right)^{-2N}\left|\cc\left(\lambda \right)  \right|^{-2} \text{ is in  } L^1(\mathfrak{a}^*).
\]
As a consequence of Proposition \ref{conv}, without loss of generality, we can take the sequence $\left(C_{N,j}\right)_j$ to be uniformly bounded, i.e., there exist a positive constant $C$ such that $C_{N,j} \leq C$, for all $j$. Then by the Lebesgue dominated convergence Theorem, for $r>n$, we have
\begin{equation}
\lim_{j\longrightarrow \infty}\int_{\mathfrak{a}^*}\left|{\mathcal{H}}(\varrho_{a_1,\cdots,a_r}^{j})\left( \lambda\right)\right|^2
\left|\cc\left(\lambda \right)  \right|^{-2} d\,\lambda = \int_{\mathfrak{a}^*}\left| \prod_{i=1}^r \varphi_{\lambda}(a_i^{-1})\right|^2
\left|\cc\left(\lambda \right)  \right|^{-2} d\,\lambda. \label{eq-lim}
\end{equation}
Using (\ref{eq-lim}) and Fatou's Lemma, we get	
	\begin{align}
	\int_G \left| \varrho_{a_1,\cdots,a_r}(g)\right|^2 d\mu_G(g) & \leq \liminf_{j \rightarrow \infty}  \int_G \left| \varrho^j_{a_1,\cdots,a_r}(g)\right|^2 d\mu_G(g)\\  & \leq \frac{1}{\left| \mathcal{W}\right|}  \bigintssss_{{\mathfrak{a}}^*} \left(\left|\cc (\lambda)\right|^{-1}   \prod_{i=1}^r \left|\varphi_{\lambda}(a_i^{-1})\right|\right)^2  d\lambda. \label{planch1}
	\end{align}
Combining Corollary $\ref{cor-est}$ and the estimate ($\ref{harish-chandra-c-function}$), we get
\begin{align*}
	\bigintssss_G \left| \varrho_{a_1,\cdots,a_r}(g)\right|^2 d\mu_G(g) &\leq  \frac{1}{\left| \mathcal{W}\right|}\bigg(\prod_{i=1}^{r} \left(\widetilde{C(a_i)} \right)  \bigg)^2 {\bigintss_{{\mathfrak{a}}^*}} \big(1+\left\| \lambda \right\|  \big)^{ n-l }  \Bigg(\prod_{\alpha \in \Sigma^{+}} \bigg(1+\left| \left<\lambda, \alpha\right>\right|\bigg)^{- m(\alpha)} \Bigg)^r d\lambda\\
	&\leq  C(a) \bigintssss_{{\mathfrak{a}}^*}\big(1+\left\| \lambda \right\|  \big)^{n-l} \prod_{\alpha \in \Sigma^{+}} \bigg(1+\left| \left<\lambda, \alpha\right>\right|\bigg)^{- rm(\alpha)} \,d\lambda, 
\end{align*}
where \[C(a)=C\left(a_1,...,a_r \right) = \frac{1}{\left| \mathcal{W}\right|}\bigg(\prod_{i=1}^{r} \left(\widetilde{C(a_i)} \right)  \bigg)^2 \]
is a constant which depends only on the points $a_1,...,a_r$. 

\medskip

 From the inequality (\ref{ineq-import}), we deduce that
$$
\bigintssss_{{\mathfrak{a}}^*}\big(1+\left\| \lambda \right\|  \big)^{n-l} \prod_{\alpha \in \Sigma^{+}} \bigg(1+\left| \left<\lambda, \alpha\right>\right|\bigg)^{- r\,m_{\alpha}} \, d\lambda
$$
$$
= 
\bigintss_{0}^{\infty}t^{l-1}\Bigg(\bigintss_{\mathbb{S}^{l-1}} \frac{\big(1+\left\| t\xi \right\|  \big)^{n-l}}{\prod_{\alpha \in \Sigma^{+}} \bigg(1+\left| \left<t\xi, \alpha\right>\right|\bigg)^{ r\,m_{\alpha}}} d\sigma(\xi)\Bigg)d\,t 
$$
$$
= 
\bigintss_{0}^{\infty}t^{l-1}(1+t)^{n-l}\Bigg(\bigintss_{\mathbb{S}^{l-1}} \frac{d\sigma(\xi)}{\prod_{\alpha \in \Sigma^{+}} \bigg(1+t\left| \left<\xi, \alpha\right>\right|\bigg)^{ r\,m_{\alpha}}} \Bigg)d\,t
$$
$$
\leq C
\bigintss_{0}^{\infty}\frac{t^{l-1}(1+t)^{n-l}}{(1+ct)^{r  \min \,{m_{\alpha_i}}}}d\,t\\
\leq C
\bigintss_{0}^{\infty}\frac{  t^{l-1} \, (1+t)^{n-l}  }{  (1+ct)^{r} }d\,t,
$$
where $d\,\sigma$ is the induced Lebesgue measure on the unit sphere $\mathbb{S}^{l-1}$. The Theorem follows from the fact that the integral 
\[
\bigintss_{0}^{\infty}\frac{t^{l-1}(1+t)^{n-l}}{(1+ct)^{r}}d\,t
\]
 is convergent for
\[
r> n.
\]
\end{proof}

\section{$C^k$-regularity of the Radon-Nikodym derivative} \label{$C^k$-regularity}

The aim of this section is to prove part $(2)$ of the Main Theorem.
As a consequence of Lemma \ref{conv}, we have the following
\begin{proposition} \label{conv-der} Let $X$ be an element of $\mathfrak{g}$. If $r > n+1$, then
	\[
	X\varrho^j_{a_1,\cdots,a_r}(g) \longrightarrow X\varrho_{a_1,\cdots,a_r}(g).
	\]
\end{proposition}

\begin{proof}
Applying the inversion formula for the spherical transform, Theorem \ref{inversion}, to the function $\varrho_{a_1,\cdots,a_r}^{j}$, we get
\begin{align*}
\varrho^j_{a_1,\cdots,a_r}(g)
&= c \, \bigintssss_{{\mathfrak{a}}^*} {\mathcal{H}\big(\varrho^j_{a_1,\cdots,a_r}\big)(\lambda)} \varphi_{\lambda}\left( g\right)\left|\cc\left(\lambda \right)  \right|^{-2} d\,\lambda,
\end{align*}
where $c$ is a constant independent of the function $\varrho^j_{a_1,\cdots,a_r}$.\\

Hence, for $X$ in $\mathfrak{g}$, we have,
\begin{align*}
X\varrho^j_{a_1,\cdots,a_r}(g) &= 
\frac{d}{dt}\varrho^j_{a_1,\cdots,a_r}\big(g\exp(tX)\big)_{\mid t=0} \\
&=
c \,\frac{d}{dt} \bigg( \bigintssss_{{\mathfrak{a}}^*} {\mathcal{H}\big(\varrho^j_{a_1,\cdots,a_r}\big)(\lambda)} \varphi_{\lambda}\left( g\exp(tX)\right)\left|\cc\left(\lambda \right)  \right|^{-2} d\,\lambda,\bigg)_{\mid \, t=0}\\
&=
c \, \bigintssss_{{\mathfrak{a}}^*} {\mathcal{H}\big(\varrho^j_{a_1,\cdots,a_r}\big)(\lambda)} \frac{d}{dt}\bigg(\varphi_{\lambda}\left( g\exp(tX)\right) \bigg)_{\mid \, t=0}\left|\cc\left(\lambda \right)  \right|^{-2} d\,\lambda.
\end{align*} 
The justification of the derivation under the integral sign follows the same argument as in section \ref{L2-regularity} and uses Gangolli-Varadarajan estimate (see \cite{Anker}, Proposition 3, (vi)).\\

From Lemma \ref{conv}, we know that
\[\tiny{ 
 {\mathcal{H}\big(\varrho^j_{a_1,\cdots,a_r}\big)(\lambda)} \frac{d}{dt}\bigg(\varphi_{\lambda}\left( g\exp(tX)\right) \bigg)_{\mid \, t=0}\left|\cc\left(\lambda \right)  \right|^{-2} \rightarrow  {\mathcal{H}\big(\varrho_{a_1,\cdots,a_r}\big)(\lambda)} \frac{d}{dt}\bigg(\varphi_{\lambda}\left( g\exp(tX)\right) \bigg)_{\mid \, t=0}\left|\cc\left(\lambda \right)  \right|^{-2}}.
\]
Moreover, as a consequence of Proposition $\ref{widehatofprod}$ we get 
\begin{align*}
\mathcal{H}\big(\varrho_{a_1,\cdots,a_r}\big)(\lambda) \frac{d}{dt}\bigg(\varphi_{\lambda}\left( g\exp(tX)\right) \bigg)_{\mid \, t=0}\left|\cc\left(\lambda \right)  \right|^{-2} &=\prod_{i=1}^{r} \varphi_{\lambda}(a_i^{-1}) \frac{d}{dt}\bigg(\varphi_{\lambda}\left( g\exp(tX)\right) \bigg)_{\mid \, t=0}\left|\cc\left(\lambda \right)  \right|^{-2}.
\end{align*}
With the help of Gangolli-Varadarajan estimate (see \cite{Anker}, Proposition 3, (vi)), we get
$$
\left| \prod_{i=1}^r \varphi_{s \xi}(a_i^{-1}) \frac{d}{dt}\bigg(\varphi_{s \xi}\left( g\exp(tX)\right) \bigg)_{\mid \, t=0}\left|\cc\left(s \xi \right)  \right|^{-2}  \right| \leq C \, \frac{s^{l-1} (1 + s) (1 + s)^{n-l}}{(1 + cs)^r}, 
$$
where $C > 0$ is independent of $g$, and since for $r>n+1$, the function
\[
\prod_{i=1}^r \varphi_{s \xi}(a_i^{-1}) \frac{d}{dt}\bigg(\varphi_{s \xi}\left( g\exp(tX)\right) \bigg)_{\mid \, t=0}\cc\left(s \xi \right)^{-2} \,  \in L^1(\mathfrak{a}^*).
\]	
The Proposition follows from the Paley-Wiener Theorem for spherical functions on semisimple Lie groups (\cite{Gangolli}, Theorem 3.5) and the Lebesgue dominated convergence theorem.
\end{proof}
Let $r>n+1$. Using Proposition \ref{conv-der}, by passing to the limit, we get 
\begin{align*}
X\varrho_{a_1,\cdots,a_r}(g) 
&= c\bigintssss_{{\mathfrak{a}}^*} {\mathcal{H}\big(\nu_{1}\ast...\ast\nu_{r}\big)(\lambda)} \frac{d}{dt}\bigg(\varphi_{\lambda}\left( g\exp(tX)\right) \bigg)_{\mid \, t=0}\left|\cc\left(\lambda \right)  \right|^{-2} d\,\lambda, \\
&= c \, \bigintssss_{{\mathfrak{a}}^*} \prod_{i=1}^r \varphi_{\lambda}(a_i^{-1}) \frac{d}{dt}\bigg(\varphi_{\lambda}\left( g\exp(tX)\right) \bigg)_{\mid \, t=0}\left|\cc\left(\lambda \right)  \right|^{-2} d\,\lambda \big(\text{    by Proposition \ref{widehatofprod}}\big), \\
&= \int_0^\infty \int_{\mathbb{S}^{l-1}} \prod_{i=1}^r \varphi_{s \xi}(a_i^{-1}) \frac{d}{dt}\bigg(\varphi_{s \xi}\left( g\exp(tX)\right) \bigg)_{\mid \, t=0}\left|\cc\left(s \xi \right)  \right|^{-2} \, d \sigma(\xi) \, ds.
\end{align*}
Since the function 
$$
g \longmapsto \prod_{i=1}^r \varphi_{s \xi}(a_i^{-1}) \frac{d}{dt}\bigg(\varphi_{s \xi}\left( g\exp(tX)\right) \bigg)_{\mid \, t=0}\left|\cc\left(s \xi \right)  \right|^{-2}
$$
is infinitely differentiable (see \cite[Proposition 3]{Anker}) and  
$$
\left| \prod_{i=1}^r \varphi_{s \xi}(a_i^{-1}) \frac{d}{dt}\bigg(\varphi_{s \xi}\left( g\exp(tX)\right) \bigg)_{\mid \, t=0}\left|\cc\left(s \xi \right)  \right|^{-2} \right| \leq C \, \frac{s^{l-1} (1 + s) (1 + s)^{n-l}}{(1 + cs)^r}, 
$$
where $C > 0$ is independent of $g$, the differentiation under integral sign theorem implies that if $r > n+1$, then
$$
\varrho_{a_1,\cdots,a_r} \in C^1 (G).
$$
We conclude by iteration that if $r > n + k$, then
$$
\varrho_{a_1,\cdots,a_r} \in C^k (G).
$$
\section{Case of an Arbitrary Symmetric Space of Noncompact Type}\label{general-case}

 Suppose that $G/K$ is an arbitrary symmetric space of noncompact type, not necessarily irreducible. Then by \cite{Ochiai}, each irreducible factor in the de Rham decomposition of $G/K$ is again
a Riemannian symmetric space of the noncompact type. Hence we can write
\begin{equation}
G/K =G_1/K_1 \times \cdots \times G_s/K_s,
\end{equation}
where $G_i/K_i$, $i=1,...,s$ are irreducible symmetric spaces of noncompact type.
Fix left Haar measures $\mu_{G_i}$ on $G_i$, $i=1,...,s$. Then 
\[
\mu_G=\mu_{G_1} \times \cdots \times \mu_{G_s}
\]
is a left Haar measure on $G$.\\
 Let 
\[
a_i=\left(a_i^1,\cdots, a_i^s \right)
\]
be an element of $G$ and assume that $a_i^j \not \in N_{G_j}(K_j)$, where $N_{G_j}(K_j)$ is the normalizer of $K_j$ in $G_j$, $j=1,\cdots, s$. Then it can be seen that the measures $\nu_{a_i}$, defined in section \ref{introduction}, can be written as
\[
\nu_{a_i}=\nu_{a_i^1}\times \cdots \times \nu_{a_i^s}, 
\]
where $\nu_{a_i^j}$ is the measure corresponding to the linear functional
\[
{\mathscr{I}}_{a_i^j}(f)=\int_{K_j} \int_{K_j} f(k_1a_i^jk_2)  d\mu_{K_j}(k_1)d\mu_{K_j}(k_2),
\]
where $\mu_{K_j}$ a fixed Haar measure on $K_j$, and $f$ is a continuous function with compact suppost on $G_j$. 

Applying Ragozin's result to each component of the de Rham decomposition of $G/K$, we deduce that if
\[
r \geq \max_{1\leq i \leq s}{\dim G_i/K_i},
\]
then $ \nu_{a_1^j}\ast...\ast\nu_{a_r^j}$ is absolutely continuous with respect to $\mu_{G_j}$ for $j=1, \cdots s$. If we denote by $\varrho_{a_1^j,\cdots,a_r^j}$ the Radon-Nikodym derivative of $\nu_{a_1^j}\ast...\ast\nu_{a_r^j}$ with respect to the Haar measure $\mu_{G_j}$, then
\[
\varrho_{a_1^j,\cdots,a_r^j}= \frac{d\left( \nu_{a_1^j}\ast...\ast\nu_{a_r^j}\right) }{d\mu_{G_j}} \in  L^{1}(G_j).
\]
Since
\[
\nu_{a_1}\ast...\ast\nu_{a_r}= \left( \nu_{a_1^1}\ast...\ast\nu_{a_r^1}\right)  \times \cdots \times \left( \nu_{a_1^s}\ast...\ast\nu_{a_r^s}\right) , 
\]
we deduce that if $r \geq \max_{1\leq i \leq s}{\dim G_i/K_i}$ then the measure $\nu_{a_1}\ast...\ast\nu_{a_r}$ is absolutely continuous with respect to $\mu_G=\mu_{G_1} \times \cdots \times \mu_{G_s}$. If we denote by $\varrho_{a_1,\cdots,a_r}$ the Radon-Nikodym derivative of $\nu_{a_1}\ast...\ast\nu_{a_r}$ with respect to $\mu_G$, then 
\begin{align*}
\varrho_{a_1,\cdots,a_r}\left(x_1,\cdots, x_s \right) &=\frac{d\left(\nu_{a_1}\ast...\ast\nu_{a_r} \right) }{d\mu_G}\\
&=
\frac{d\bigg(\left( \nu_{a_1^1}\ast...\ast\nu_{a_r^1}\right)  \times \cdots \times \left( \nu_{a_1^s}\ast...\ast\nu_{a_r^s}\right)\bigg)}{d\left(\mu_{G_1} \times \cdots \times \mu_{G_s} \right) }\\
&=\frac{d\left( \nu_{a_1^1}\ast...\ast\nu_{a_r^1}\right) }{d\mu_{G_1}}\left(x_1 \right)  \cdots  \frac{d\left( \nu_{a_1^s}\ast...\ast\nu_{a_r^s}\right) }{d\mu_{G_s}}\left(x_s \right)   \\
&=\varrho_{a_1^1,\cdots,a_r^1}(x_1) \cdots \varrho_{a_1^s,\cdots,a_r^s}(x_s) \in L^1(G).
\end{align*}
Hence
\begin{align}
\int_G \left| \varrho_{a_1,\cdots,a_r}\right|^2d\,\mu_G &=\prod_{i=1}^s \int_{G_i}\left|\varrho_{a_1^i,\cdots,a_r^i} \right|^2d\,\mu_{G_i}. \label{produit} 
\end{align}
As a consequence of $(\ref{produit})$ and part (i) of the Main Theorem, we deduce that if
\[
r > \max_{1\leq i \leq s}{\dim G_i/K_i}, 
\]
 then
\[
\varrho_{a_1,\cdots, a_r} \in L^2(G).
\]
Similarly, applying part (ii) of the Main Theorem, we deduce that if
\[
r>\max_{1\leq i \leq s}{\dim G_i/K_i}+k,
\]
then
\[
\varrho_{a_1,\cdots,a_r} \in C^k(G).\\
\]

\begin{acknowledgment*}
It is a great pleasure for me to thank Ka\"is Ammari for helpful suggestions and his careful reading of the paper. It's also a great pleasure to thank Mahmoud Al-Hashami for pointing out some misprints in an earlier version of the paper. 
\end{acknowledgment*}

\end{document}